\documentclass[a4paper]{article}

\usepackage{amsfonts, latexsym, amssymb, amsthm, amsmath, mathrsfs, esint}
\usepackage{authblk}
\usepackage{hyperref}

\newcommand{\R}{\mathbb{R}}
\newcommand{\N}{\mathbb{N}}

\newcommand{\W}{\mathcal{W}}

\newcommand{\loc}{\mathrm{loc}}

\let\div\relax\DeclareMathOperator{\div}{div}
\DeclareMathOperator*{\esssup}{ess\,sup}
\DeclareMathOperator{\dist}{dist}

\DeclareMathOperator{\supp}{supp}
\newcommand{\blank}{{\mkern 2mu\cdot\mkern 2mu}}

\newcommand{\set}[2]{\left\{ #1 \mid #2 \right\}}

\newtheorem{theorem}{Theorem}
\newtheorem{lemma}[theorem]{Lemma}
\newtheorem{proposition}[theorem]{Proposition}
\newtheorem{definition}[theorem]{Definition}

\theoremstyle{remark}
{}

\title{Variational problems in $L^\infty$ involving semilinear second order differential operators}

\author{Nikos Katzourakis \\ \small Department of Mathematics and Statistics, University of Reading, Whiteknights, Pepper Lane, Reading, RG6 6AX, UK. E-mail: n.katzourakis@reading.ac.uk
\and Roger Moser \\ \small Department of Mathematical Sciences,
University of Bath, Bath BA2 7AY, UK.
E-mail: r.moser@bath.ac.uk}

\begin{document}

\maketitle

\begin{abstract}
For an elliptic, semilinear differential operator of the form $S(u) = A : D^2 u + b(x, u , Du)$,
consider the functional $E_\infty(u) = \esssup_\Omega |S(u)|$.
We study minimisers of $E_\infty$ for prescribed boundary data. Because the functional is not differentiable,
this problem does not give rise to a conventional Euler-Lagrange equation. Under certain conditions,
we can nevertheless give a system of partial differential equations that all minimisers must
satisfy. Moreover, the condition is equivalent to a weaker version of the variational problem.
\end{abstract}

\section{Introduction}

Variational problems involving an $L^\infty$-norm tend to present challenges
not shared by more conventional variational problems. Indeed, the underlying
functionals are typically not differentiable, not even in the Gateaux sense,
and therefore, the usual derivation of an Euler-Lagrange equation does not
work. Sometimes it is possible to derive an associated partial differential
equation nevertheless (the Aronsson equation \cite{Aronsson:68} is an example) but such
an equation is typically only degenerate elliptic and may have
discontinuous coefficients \cite{Katzourakis:12}. Indeed, for higher order problems (as
studied in this paper), the equations may be fully nonlinear and
not elliptic at all \cite{Katzourakis-Pryer:20}. Moreover, while the
interesting functionals in the calculus
of variations in $L^\infty$ typically enjoy a certain degree of convexity,
they are not strictly convex. Therefore, minimisers are not
usually expected to be unique.

All these difficulties notwithstanding, there are some problems in the theory
that are understood very well. This applies in particular to a class
of problems involving first order derivatives of scalar functions.
More precisely, let $\Omega \subseteq \R^n$ be a bounded Lipschitz domain
and consider functions $u \colon \Omega \to \R$ with fixed boundary data.
For a function $F \colon \Omega \times \R \times \R^n \to \R$, consider the
problem of minimising $\esssup_{x \in \Omega} F(x, u(x), Du(x))$.
Under certain conditions on $F$, there is a good theory giving, for
example, existence of solutions with good properties
\cite{Aronsson:67, Barron-Jensen-Wang:01.1, Barron-Jensen-Wang:01.2},
uniqueness \cite{Jensen:93, Juutinen:98}, and (for $F(x, y, z) = |z|^2$)
regularity \cite{Evans-Savin:08, Evans-Smart:11.1, Evans-Smart:11.2}.

Recently, the authors have studied a different, second order
variational problem in $L^\infty$ and established good properties of its solutions \cite{Katzourakis-Moser:19}. Suppose now that we wish to minimise
a quantity such as $\esssup_{x \in \Omega} F(x, \Delta u(x))$. Assuming
that $F \colon \Omega \times \R \to \R$ satisfies some
growth and convexity conditions, and that $\Omega$, $F$, and the boundary
data are sufficiently regular, it turns out that there exists a unique minimiser,
which satisfies a certain system of partial differential equations.
Conversely, any solution of that system corresponds to a minimiser.
Some of the underlying ideas go back to earlier work \cite{Moser-Schwetlick:12, Sakellaris:17},
and similar tools have in the meantime also been used for other problems
\cite{Katzourakis-Parini:17, Moser:19, Gallagher-Moser:22}.

In the current paper, we study extensions of these results. This is one of a
pair of works that examine two different types of generalisations.
Here we replace $\Delta u$ by more general, semilinear differential
operators, while restricting our attention to $F(x, \xi) = |\xi|$ (but
$x$-dependence is still included implicitly, because the coefficients
of the differential operator need not be constant any more). In a companion paper
\cite{Katzourakis-Moser:22}, we study a quantity of the form
$\esssup_{x \in \Omega} F(x, u(x), \Delta u(x))$
for a fairly general class of functions $F \colon \Omega \times \R^2 \to \R$.
There is of course some overlap between the two settings, and it would
be interesting to have a common framework, but this appears to be
difficult for technical reasons.

We consider the following situation. Let $\Omega \subseteq \R^n$, as above,
be a bounded Lipschitz domain. (For some of our results we will need to
impose additional regularity assumptions on $\Omega$.)
Let $A \in C^2(\overline{\Omega}; \R^{n \times n})$ and
$b \in C^2(\overline{\Omega} \times \R \times \R^n)$. We assume that
there exists $\lambda > 0$ such that for every $x \in \Omega$, the
matrix $A(x)$ is symmetric and satisfies $A(x) : \zeta \otimes \zeta \ge \lambda |\zeta|^2$ for all $\zeta \in \R^n$,
where the colon denotes the Frobenius inner product. Define the semilinear differential operator
\begin{equation} \label{eq:differential-operator}
S(u) = A : D^2 u + b(x, u, Du)
\end{equation}
for $u \colon \Omega \to \R$, where $Du$ is the gradient and $D^2 u$ is
the Hessian of $u$. We are interested in the functional
\[
E_\infty(u) = \esssup_\Omega |S(u)|.
\]

So far we have not mentioned the space on which this functional is defined.
In order to obtain a good theory, we need to work in a Sobolev space that
may appear unconventional, but is quite natural for the problem. We define
\[
\W^{2, \infty}(\Omega) = \bigcap_{1 < q < \infty} \set{u \in W^{2, q}(\Omega)}{A : D^2 u \in L^\infty(\Omega)}.
\]
Furthermore,
\[
\W_0^{2, \infty}(\Omega) = \W^{2, \infty}(\Omega) \cap W_0^{2, 2}(\Omega).
\]

Given $u_0 \in \W^{2, \infty}(\Omega)$, we wish to study minimisers of
$E_\infty$ in $u_0 + \W_0^{2, \infty}(\Omega)$. For other variational
problems, the critical points would also be of interest, but in this
case, the concept is not useful as $E_\infty$ is not differentiable.
We therefore work with the following idea instead.

\begin{definition}[Almost-minimiser] \label{def:almost-minimiser}
A function $u \in \W^{2, \infty}(\Omega)$ is called an \emph{al\-most-minimiser}
of $E_\infty$ if there exists $M \in \R$ such that
\[
E_\infty(u) \le E_\infty(u + \phi) + M \|\phi\|_{W^{1, \infty}(\Omega)}^2
\]
for every
$\phi \in \W_0^{2, \infty}(\Omega)$.
\end{definition}

Intuitively, the definition provides a substitute for the idea that the
Taylor expansion of $E_\infty$ has a vanishing first order term at $u$. Since
$E_\infty$ is not differentiable, it does of course not have a Taylor expansion.
Instead, we use a quadratic form $Q \colon \W^{2, \infty}(\Omega) \to \R$ such that
the graph of $Q$ touches the graph of $E_\infty(\blank + u) - E_\infty(u)$ from below.
This is related to the notion of second order subjets or subdifferentials that appears in the
theory of viscosity solutions of partial differential equations
(discussed, e.g., in a survey article by Crandall, Ishii, and Lions \cite{Crandall-Ishii-Lions:92}
or the introductory text by the first author \cite{Katzourakis:15.1}).
In the case of Definition \ref{def:almost-minimiser}, the condition
is formulated in terms of the norm of $W^{1, \infty}(\Omega)$, so it may
be best to think of $E_\infty$ as a functional defined on $W^{1, \infty}(\Omega)$
in this context, with $E_\infty(v) = \infty$ when $v \not\in \W^{2, \infty}(\Omega)$.
There is, however, some flexibility here. Our main results remain true
if $\|\blank\|_{W^{1, \infty}(\Omega)}$
is replaced, e.g., by $\|\blank\|_{W^{2, q}(\Omega)}$ for any $q < \infty$.

We show below that almost-minimisers can be characterised in terms of a
system of partial differential equations. In order to formulate this,
we use the formal linearisation of the operator $S$ at a point
$u \in \W^{2, \infty}(\Omega)$, denoted $S_u'$, and its formal $L^2$-adjoint,
denoted $S_u^*$. We use the notation $(x, y, z)$ for the variables
in $\Omega \times \R \times \R^n$. We further write $b_y$ for the partial
derivative of $b$ with respect to $y$, and $b_z$ for the vector comprising
the partial derivatives with respect to $z_1, \dotsc, z_n$. Then
\[
S_u' \phi = A : D^2\phi + b_z(x, u, Du) \cdot D\phi + b_y(x, u, Du)\phi
\]
and
\[
S_u^* f = \div \div (f A) - \div\bigl(f b_z(x, u, Du)\bigr) + f b_y(x, u, Du).
\]
In the div-div term, the divergence is applied once column-wise and once
row-wise.

We are interested in the equation
\begin{equation} \label{eq:S_u^*f=0}
S_u^* f = 0.
\end{equation}
We can make sense
of this for $f \in L^1(\Omega)$: if
\begin{equation} \label{eq:weak-solution}
\int_\Omega f S_u'\phi \, dx = 0
\end{equation}
holds for all $\phi \in C_0^\infty(\Omega)$, then we say that $f$ is a
\emph{weak solution} of the equation.

Our first main result is as follows.

\begin{theorem} \label{thm:sufficiency}
Let $u_\infty \in \W^{2, \infty}(\Omega)$ be
such that there exist $e_\infty \ge 0$ and $f_\infty \in L^1(\Omega) \setminus \{0\}$
satisfying
\begin{equation} \label{eq:Euler-Lagrange1}
|f_\infty| S(u_\infty) = e_\infty f_\infty
\end{equation}
almost everywhere in $\Omega$ and
\begin{equation} \label{eq:Euler-Lagrange2}
S_{u_\infty}^* f_\infty = 0
\end{equation}
weakly. Then $u_\infty$ is an almost-minimiser of $E_\infty$.
\end{theorem}

The converse is also true, provided that we impose some additional regularity
on $\partial \Omega$ and on the boundary data, and provided that we
restrict our attention to differential operators that permit certain
$L^p$-estimates.

\begin{definition} \label{def:admissible}
For a differential operator $S$ as in \eqref{eq:differential-operator},
we say that $S$ is \emph{admissible} if there exists $p_0 > 1$ such that
for any $p \ge p_0$, the following holds true. Suppose that
$u_0 \in \W^{2, \infty}(\Omega)$ and $\Lambda > 0$. Then there exists
$C > 0$ such that for any $u \in u_0 + \W_0^{2, p}(\Omega)$, if
\[
\|S(u)\|_{L^p(\Omega)} \le \Lambda,
\]
then
\[
\|u\|_{W^{2, p}(\Omega)} \le C.
\]
\end{definition}

Now we can formulate our second main result.

\begin{theorem} \label{thm:necessity}
Suppose that $\partial \Omega$ is of class $C^3$ and $u_0 \in C^2(\overline{\Omega})$.
Further suppose that $S$ is admissible.
If $u_\infty \in u_0 + \W_0^{2, \infty}(\Omega)$ is an almost-minimiser of $E_\infty$, then there exist $f_\infty \in L^1(\Omega) \setminus \{0\}$ and $e_\infty > 0$ such that \eqref{eq:Euler-Lagrange1} holds almost everywhere in $\Omega$ and \eqref{eq:Euler-Lagrange2} holds weakly.
\end{theorem}

For admissible operators, we can furthermore be
certain that minimisers of $E_\infty$ exist for prescribed boundary data. A minimiser
is in particular an almost-minimiser, and thus we are guaranteed that
the system \eqref{eq:Euler-Lagrange1}, \eqref{eq:Euler-Lagrange2} has a
non-trivial solution.

\begin{proposition} \label{prp:minimisers}
If $S$ is admissible, then $E_\infty$ attains its minimum in
$u_0 + \W_0^{2, \infty}(\Omega)$ for any $u_0 \in \W^{2, \infty}(\Omega)$.
\end{proposition}

The proof relies on standard arguments and in particular on the
direct method. For completeness, we outline
these arguments anyway.

\begin{proof}
Let $(u_k)_{k \in \N}$ be a minimising sequence. Then $\|S(u_k)\|_{L^\infty(\Omega)}$ is obviously bounded. Hence if $p_0$ is the number from Definition \ref{def:admissible}, then
it follows that $(u_k)_{k \in \N}$ is bounded in $W^{2, p}(\Omega)$ for any $p \in [p_0, \infty)$.
Therefore, we may assume (extracting a subsequence if necessary) that we have the convergence
$u_k \rightharpoonup u_\infty$ weakly in $W^{2, p}(\Omega)$ for every $p < \infty$. Moreover, the limit belongs to $u_0 + W_0^{2, p}(\Omega)$.

Using the Sobolev embedding theorem, we further conclude that $(u_k)_{k \in \N}$ is bounded in $C^{1, \alpha}(\overline{\Omega})$ as well for every $\alpha \in (0, 1)$. The Arzel\`a-Ascoli theorem implies that $u_k \to u_\infty$ in $C^1(\overline{\Omega})$. Hence $b(x, u_k, Du_k) \to b(x, u_\infty, Du_\infty)$ uniformly.

Since $(S(u_k))_{k \in \N}$ is bounded in $L^\infty(\Omega)$, we may assume that
$S(u_k) \stackrel{*}{\rightharpoonup} \sigma$ weakly* in $L^\infty(\Omega)$ for some $\sigma \in L^\infty(\Omega)$. Using the above convergence, we conclude that $\sigma = S(u_\infty)$. Now it follows that
\[
E_\infty(u_\infty) \le \liminf_{k \to \infty} E_\infty(u_k),
\]
and thus $u_\infty$ is a minimiser.
\end{proof}

Summarising Theorem \ref{thm:sufficiency} and Theorem \ref{thm:necessity},
we can say that the system comprising equations \eqref{eq:Euler-Lagrange1}
and \eqref{eq:Euler-Lagrange2} is \emph{equivalent} to the almost-minimising
condition under certain assumptions. If we accept that the latter is a
reasonable substitute for critical points, then we may think of
\eqref{eq:Euler-Lagrange1} and \eqref{eq:Euler-Lagrange2} as a substitute
for an Euler-Lagrange equation.

For a variety of other variational problems in $L^\infty$, a corresponding differential equation
has been identified by quite different methods. In the case of the
optimal Lipschitz extension problem, the result is the Aronsson equation
\cite{Aronsson:68}.
Formally, Aronsson's calculations can be carried out for the above problem
as well. They give rise to the equation $S(u) D(S(u)) = 0$.
The connections between this equation and the variational problem are not explored in this work,
but we observe that the former is satisfied by solutions of \eqref{eq:Euler-Lagrange1}.

The `Aronsson equation' $S(u)D(S(u))=0$ is of third order, quasilinear, and in non-divergence form. It is
not elliptic in any reasonable sense. It allows neither weak nor viscosity solutions, but
there is another approach that does apply and has produced some results on equations such as this
(see, e.g., the papers of the first author and coauthors \cite{Katzourakis:17.2, Croce-Katzourakis-Pisante:17, Katzourakis-Pryer:20}).
In this paper, however, we follow the alternative approach outlined above and consider solutions to the system
\eqref{eq:Euler-Lagrange1}, \eqref{eq:Euler-Lagrange2} instead. This can be seen as a divergence-form
(or div-div form) alternative to the Aronsson equation.

The comparison with other variational problems in $L^\infty$ makes
another aspect of the above results remarkable. We note that the system
\eqref{eq:Euler-Lagrange1}, \eqref{eq:Euler-Lagrange2} is local in the
sense that if it holds in $\Omega$, then it is automatically satisfied
in any open subdomain $\Omega' \subseteq \Omega$. Under the conditions of Theorem \ref{thm:necessity},
it follows that the almost-minimising condition is also local in
a similar sense. Indeed, if $\partial \Omega$ is of class $C^3$ and $u_0 \in C^2(\overline{\Omega})$,
then for any almost-minimiser $u_\infty \in u_0 + \W_0^{2, \infty}(\Omega)$, Theorem \ref{thm:necessity}
applies. Given a Lipschitz subdomain $\Omega' \subseteq \Omega$,
we can then use Theorem \ref{thm:sufficiency} in $\Omega'$,
concluding that $u_\infty$ is an almost-minimiser in $\Omega'$ as well.
This is in stark contrast to the optimal Lipschitz extension
problem and similar variational problems, where locality must be imposed
in order to obtain solutions with good properties. These solutions are
then known as absolute minimisers.

For the special case $S(u) = \Delta u$ (and for certain other problems),
a previous paper \cite{Katzourakis-Moser:19} gives stronger results.
It is shown that solutions of \eqref{eq:Euler-Lagrange1}, \eqref{eq:Euler-Lagrange2} are
unique under the boundary condition $u_\infty \in u_0 + \W^{2, \infty}(\Omega)$
and correspond to unique minimisers of $E_\infty$.
For similar problems involving nonlinear operators, however, we do not have uniqueness in general \cite{Moser:19}.
For semilinear operators as in this paper, the question is open.

In the next section, we discuss second order elliptic, linear equations
in div-div form, of which \eqref{eq:Euler-Lagrange2} is an example.
We consider solutions in $L^1(\Omega)$ and derive some interior regularity
that we will need for the proofs of our main results. We will also show that
weak solutions can be tested with functions in $\W_0^{2, \infty}(\Omega)$.
We then prove Theorem \ref{thm:sufficiency} and Theorem \ref{thm:necessity}
in the following two sections. In the final section, we discuss some specific
differential operators of the form $S(u) = \Delta u + g(u)$ for a function
$g \colon \R \to \R$. In particular, we give some conditions that imply
that the operator is admissible in the sense of Definition \ref{def:admissible}.
The purpose of this section is not just to show that Theorem
\ref{thm:necessity} is not vacuous, but also to give an idea of the
nonlinearities allowed.

\section{Elliptic equations in div-div form}

In this section, we prove some properties of weak
solutions of an equation of the form
\begin{equation} \label{eq:div-div}
\div \div (fA) + \div(fB) + fc = \div G + g,
\end{equation}
where $f$ is assumed to be in $L^1(\Omega)$ or even a Radon measure on $\Omega$.
The matrix $A$ will be the same as in the introduction and will be fixed throughout.
The properties of the coefficients $B$ and $c$ are described below.
We eventually apply these results to equations such as
$S_u^* f = \div G + g$ for some $u \in W^{2, p}(\Omega)$, or even to
$S_{u_\infty}^* f = 0$ for a function
$u_\infty \in \W^{2, \infty}(\Omega)$, but we formulate them more generally here.

First we prove some interior regularity for weak solutions
of the equation. We duplicate some results from a more general theory
here (see, e.g., the survey article of Bogachev, Krylov, and R\"ockner
\cite{Bogachev-Krylov-Rockner:09}). In order to make the paper self-contained,
we include a proof nevertheless.

\begin{lemma} \label{lem:regularity}
For any $\Omega' \Subset \Omega$ and any $p \in (n, \infty)$, there exists a
constant $C > 0$ such that the following holds true.
Let $p' = \frac{p}{p - 1}$ be the exponent conjugate to $p$.
Suppose that $B \in L^\infty(\Omega; \R^n)$, $c \in L^\infty(\Omega)$, $g \in L^1(\Omega)$,
and $G \in L^{p'}(\Omega; \R^n)$. Set
\[
\Gamma = \|B\|_{L^\infty(\Omega)} + \|c\|_{L^\infty(\Omega)} + 1.
\]
Suppose that $\mu \in (C_0^0(\Omega))^*$ is a distributional solution of
\[
\div \div (A\mu) + \div(B\mu) + c\mu = \div G + g,
\]
meaning that
\begin{equation} \label{eq:div-div-distributional}
\int_\Omega (A : D^2\phi - B \cdot D\phi + c\phi) \, d\mu = \int_\Omega (g\phi - G \cdot D\phi) \, dx
\end{equation}
for every $\phi \in C_0^\infty(\Omega)$. Then $\mu$ is absolutely continuous
with respect to the Lebesgue measure. Its Radon-Nikodym derivative
$f$ belongs to $W_\loc^{1, p'}(\Omega)$ and satisfies
\[
\|f\|_{W^{1, p'}(\Omega')} \le C\Gamma^2 \left(|\mu|(\Omega) + \|G\|_{L^{p'}(\Omega)} + \|g\|_{L^1(\Omega)}\right).
\]
\end{lemma}

\begin{proof}
Consider a function $\chi \in C_0^\infty(\Omega)$. Define a functional
$\alpha \in (C_0^1(\Omega))^*$ by
\[
\alpha(\psi) = \int_\Omega \chi \psi \, d\mu, \quad \psi \in C_0^1(\Omega).
\]
(This means that $\alpha$ corresponds to the measure $\chi \mu$, but we regard it as a functional
on $C_0^1(\Omega)$ at first.)
Choose an open, precompact set $\Omega'' \Subset \Omega$ with smooth boundary and with $\supp \chi \subseteq \Omega''$.
Given $\psi \in C_0^1(\Omega)$, we can solve the equation
\[
A: D^2 \phi = \psi
\]
in $W^{2, p}(\Omega'') \cap W_0^{1, p}(\Omega'')$ by \cite[Theorem 9.15]{Gilbarg-Trudinger:83}. Moreover,
\cite[Theorem 9.19]{Gilbarg-Trudinger:83} implies that
$\phi \in C^2(\Omega'')$. By approximation, we then see that \eqref{eq:div-div-distributional}
is satisfied for the test function $\chi \phi$. Thus we obtain
\[
\begin{split}
\alpha(\psi) & = \int_\Omega \chi A: D^2\phi \, d\mu \\
& = \int_\Omega \bigl(\chi B \cdot D\phi + \phi B \cdot D\chi - c \chi\phi - A : (2D\chi \otimes D\phi + \phi D^2 \chi)\bigr) \, d\mu \\
& \quad + \int_\Omega (\chi g\phi - \chi G \cdot D\phi - \phi G \cdot D\chi) \, dx.
\end{split}
\]

Note that \cite[Lemma 9.17]{Gilbarg-Trudinger:83} implies that
\[
\|\phi\|_{W^{2, p}(\Omega'')} \le C_1 \|\psi\|_{L^p(\Omega)}
\]
for a constant $C_1 = C_1(n, \Omega'', p, A)$. Hence
\[
\|\phi\|_{C^1(\overline{\Omega''})} \le C_2 \|\psi\|_{L^p(\Omega)}
\]
for a constant $C_2$ with the same dependence.
Therefore, we find a constant $C_3 = C_3(n, \Omega'', p, A, \chi)$ such that
\[
|\alpha(\psi)| \le C_3 \left(\Gamma |\mu|(\Omega) + \|g\|_{L^1(\Omega)} + \|G\|_{L^{p'}(\Omega)}\right) \|\psi\|_{L^p(\Omega)}.
\]
In particular, the functional $\alpha$ has a continuous linear extension to $L^p(\Omega)$.
It follows that there exists $\tilde{f} \in L^{p'}(\Omega)$ such that
\[
\int_\Omega \chi \psi \, d\mu = \int_\Omega \tilde{f} \psi \, dx
\]
for all $\psi \in C_0^1(\Omega)$. Since $C_0^1(\Omega)$ is dense in $C_0^0(\Omega)$, this means that $\chi \mu$ is absolutely continuous with respect
to the Lebesgue measure and $\tilde{f}$ is the Radon-Nikodym derivative. Since
these arguments work for any $\chi \in C_0^\infty(\Omega)$,
the measure $\mu$ is absolutely continuous as well and has a
Radon-Nikodym derivative $f \in L_\loc^{p'}(\Omega)$.
Choosing $\chi$ such that $\chi \equiv 1$ in $\Omega'$, we also obtain the inequality
\begin{equation} \label{eq:L^p-estimate1}
\|f\|_{L^{p'}(\Omega')} \le C_3 \left(\Gamma |\mu|(\Omega) + \|G\|_{L^{p'}(\Omega)} + \|g\|_{L^1(\Omega)}\right).
\end{equation}

We now conclude that \eqref{eq:div-div-distributional} holds
true for every $\phi \in W^{2, p}(\Omega)$ with compact support in $\Omega$.

For $i \in \{1, \dotsc, n\}$, we next consider the functional $\beta_i \in (C_0^1(\Omega))^*$ with
\[
\beta_i(\psi) = \int_\Omega f \chi \psi_{x_i} \, dx, \quad \psi \in C_0^1(\Omega)
\]
(corresponding to a distributional derivative of $\chi f$).
Given a fixed $\psi \in C_0^1(\Omega)$, solve
\[
A : D^2\phi + (\div A) \cdot D\phi = \psi_{x_i}
\]
in $W^{2, p}(\Omega'') \cap W_0^{1, p}(\Omega'')$. If we write $e_i$ for
the $i$-th standard unit vector in $\R^n$, then the equation can alternatively
be represented in the form
\[
\div(A D\phi) = \div(\psi e_i).
\]
Standard $L^p$-estimates for weak solutions thus give the estimate
\[
\|\phi\|_{W^{1, p}(\Omega'')} \le C_4 \|\psi\|_{L^p(\Omega)}
\]
for a constant $C_4 = C_4(n, \Omega'', p, A)$.
The function $\chi\phi$ is again a suitable test function for
\eqref{eq:div-div-distributional}. We obtain
\[
\begin{split}
\beta_i(\psi) & = \int_\Omega f \bigl(\chi (\div A) \cdot D\phi - A : (2D\chi \otimes D\phi + \phi D^2 \chi)\bigr) \, dx \\
& \quad + \int_\Omega f \bigl(B \cdot (\chi D\phi + \phi D\chi) - c \chi\phi\bigr) \, dx \\
& \quad + \int_\Omega (\chi g \phi - \chi G \cdot D\phi - \phi G \cdot D\chi) \, dx.
\end{split}
\]
Hence there exists a constant $C_5 = C_5(n, \Omega'', p, A, \chi)$ such that
\[
|\beta_i(\psi)| \le C_5 \left(\Gamma \|f\|_{L^{p'}(\Omega'')} + \|G\|_{L^{p'}(\Omega)} + \|g\|_{L^1(\Omega)}\right) \|\psi\|_{L^p(\Omega)}.
\]
Therefore, there exists $h_i \in L^{p'}(\Omega)$ such that
\[
\int_\Omega \chi f \psi_{x_i} \, dx = \int_\Omega h_i \psi \, dx
\]
for all $\psi \in C_0^1(\Omega)$.
This is true for $i = 1, \dots, n$, so the function $\chi f$ has weak derivatives in $L^{p'}(\Omega)$, which satisfy
\begin{equation} \label{eq:L^p-estimate2}
\left\|(\chi f)_{x_i}\right\|_{L^{p'}(\Omega)} \le C_5 \left(\Gamma \|f\|_{L^{p'}(\Omega'')} + \|G\|_{L^{p'}(\Omega)} + \|g\|_{L^1(\Omega)}\right).
\end{equation}
Since $\chi \in C_0^\infty(\Omega)$ can be chosen arbitrarily,
it follows that $f \in W_\loc^{1, p'}(\Omega)$. Moreover, we obtain
the desired inequality if we combine \eqref{eq:L^p-estimate1} with
\eqref{eq:L^p-estimate2}.
\end{proof}

We will also require the following statement, which says that a weak
solution of an equation of the form
\begin{equation} \label{eq:div-div-homogeneous}
\div \div (fA) + \div(fB) + cf = 0
\end{equation}
can be tested with functions from the space $\W_0^{2, \infty}(\Omega)$.

\begin{lemma} \label{lem:distributional-solutions}
Suppose that $f \in L^1(\Omega)$ is a weak solution of equation
\eqref{eq:div-div-homogeneous}. Then
\[
\int_\Omega f(A : D^2\phi - B \cdot D\phi + c\phi) \, dx = 0
\]
for all $\phi \in \W_0^{2, \infty}(\Omega)$.
\end{lemma}

\begin{proof}
Given $\phi \in \W_0^{2, \infty}(\Omega)$, we first construct a family of
approximations $(\phi_\epsilon)_{\epsilon \in (0, \epsilon_0]}$ in
$C_0^\infty(\Omega)$ such that $\phi_\epsilon \to \phi$ in $W^{2, q}(\Omega)$
for any $q < \infty$ and, at the same time, such that $A : D^2\phi_\epsilon$
remains bounded in $L^\infty(\Omega)$.

For this purpose, we extend $\phi$ by $0$ outside of $\Omega$.
Choose a finite open cover $\{G_1, \dotsc, G_L\}$ of
$\overline{\Omega}$ with the property that there exist $R > 0$ and
there exist open cones $C_1, \dotsc, C_L \subseteq \R^n$ such that
for any $x \in \partial \Omega \cap G_\ell$,
\[
C_\ell \cap B_R(0) \cap (x - \overline{\Omega}) = \emptyset.
\]
(Hence $(x - C_\ell) \cap B_R(x)$ is an exterior cone to $\overline{\Omega}$.)
This is possible, because $\Omega$ is a bounded Lipschitz domain.

For every $\ell = 1, \dotsc L$, choose $\eta_\ell \in C_0^\infty(C_\ell \cap B_1(0))$ with $\eta_\ell \ge 0$ and
\[
\int_{B_1(0)} \eta_\ell(x) \, dx = 1.
\]
For $\epsilon \in (0, R]$, set
\[
\eta_{\ell\epsilon}(x) = \frac{1}{\epsilon^n} \eta_\ell\left(\frac{x}{\epsilon}\right)
\]
and
\[
\phi_{\ell\epsilon} = \phi * \eta_{\ell\epsilon}.
\]
Then for $x \in \partial \Omega \cap G_\ell$,
\[
\phi_{\ell\epsilon}(x) = \int_{C_\ell \cap B_\epsilon(0) \cap (x - \Omega)} \eta_{\ell\epsilon}(y) \phi(x - y) \, dy = 0.
\]
Moreover, the function $\phi_{\ell\epsilon}$ vanishes in a neighbourhood of
any such point $x$.

Now choose a partition of unity $\chi_1, \dotsc, \chi_L$ in $\Omega$ with
$\chi_\ell \in C_0^\infty(G_\ell)$ for $\ell = 1, \dotsc, L$. Set
\[
\phi_\epsilon = \sum_{\ell = 1}^L \chi_\ell \phi_{\ell\epsilon}.
\]
Then it is clear that $\phi_\epsilon \in C_0^\infty(\Omega)$ and that
$\phi_\epsilon \to \phi$ in $W^{2, q}(\Omega)$, for any $q < \infty$, as
$\epsilon \searrow 0$.

For $x \in \Omega$, we compute
\[
\begin{split}
A(x) : D^2 \phi_{\ell\epsilon}(x) &= \int_{B_\epsilon(0)} \eta_{\ell\epsilon}(y) A(x) : D^2\phi(x - y) \, dy \\
& = \int_{B_\epsilon(0)} \eta_{\ell\epsilon}(y) (A(x) - A(x - y)) : D^2\phi(x - y) \, dy \\
& \quad + \int_{B_\epsilon(0)} \eta_{\ell\epsilon}(y) A(x - y) : D^2\phi(x - y) \, dy.
\end{split}
\]
Using an integration by parts, we find that
\begin{multline*}
\int_{B_\epsilon(0)} \eta_{\ell\epsilon}(y) (A(x) - A(x - y)) : D^2\phi(x - y) \, dy \\
\begin{aligned}
&= \int_{B_\epsilon(0)} \eta_{\ell\epsilon}(y) \div A(x - y) \cdot D\phi(x - y) \, dy \\
& \quad + \int_{B_\epsilon(0)} (A(x) - A(x - y)) : D\eta_{\ell\epsilon}(y) \otimes D\phi(x - y) \, dy.
\end{aligned}
\end{multline*}
For $y \in B_\epsilon(0)$, we have the inequality
\[
|A(x) - A(x - y)| \le \epsilon \|A\|_{C^1(\overline{\Omega})}.
\]
Hence there exists a universal constant $C_1$ such that
\begin{multline*}
\left|\int_{B_\epsilon(0)} \eta_{\ell\epsilon}(y) (A(x) - A(x - y)) : D^2\phi(x - y) \, dy\right| \\
\le C_1\|A\|_{C^1(\overline{\Omega})} \bigl(1 + \|D\eta_\ell\|_{L^1(B_1(0))}\bigr) \|D\phi\|_{L^\infty(\Omega)}.
\end{multline*}
It is clear that
\[
\left|\int_{B_\epsilon(0)} \eta_{\ell\epsilon}(y) A(x - y) : D^2\phi(x - y) \, dy\right| \le \|A : D^2 \phi\|_{L^\infty(\Omega)}.
\]
Thus we obtain a uniform estimate for $\|A : D^2\phi_{\ell\epsilon}\|_{L^\infty(\Omega)}$. A similar estimate for $\phi_\epsilon$ is then easy to prove.

It follows that there exists a sequence $\epsilon_k \searrow 0$ such that
$A : D^2\phi_{\epsilon_k} \stackrel{*}{\rightharpoonup} g$, weakly* in
$L^\infty(\Omega)$, for some $g \in L^\infty(\Omega)$. For any $\psi \in C_0^\infty(\Omega)$, we then compute
\[
\begin{split}
\int_\Omega \psi g \, dx &= \lim_{k \to \infty} \int_\Omega \psi A : D^2 \phi_{\epsilon_k} \, dx \\
&= \lim_{k \to \infty} \int_\Omega \div \div (\psi A) \phi_{\epsilon_k} \, dx \\
&= \int_\Omega \div \div (\psi A) \phi \, dx \\
&= \int_\Omega \psi A : D^2 \phi \, dx.
\end{split}
\]
It follows that $g = A : D^2\phi$. It then also follows that $A : D^2\phi_\epsilon \stackrel{*}{\rightharpoonup} A : D^2\phi$, i.e, it
is not necessary to take a subsequence.

Now the claim of the lemma is proved by approximation with $\phi_\epsilon$
and with standard arguments.
\end{proof}

\section{Sufficiency of the equations}

We now show that a non-trivial solution of the system \eqref{eq:Euler-Lagrange1},
\eqref{eq:Euler-Lagrange2} gives rise to an almost-minimiser of $E_\infty$.

\begin{proof}[Proof of Theorem \ref{thm:sufficiency}]
We consider $u_\infty \in \W^{2, \infty}(\Omega)$ and assume that there exist
$e_\infty \ge 0$ and $f_\infty \in L^1 \setminus \{0\}$ such that
\eqref{eq:Euler-Lagrange1} is satisfied almost everywhere and \eqref{eq:Euler-Lagrange2} weakly in $\Omega$. We are required to show that $u_\infty$
is an almost-minimiser of $E_\infty$.
It suffices, however, to prove the inequality in Definition \ref{def:almost-minimiser} under the assumption that
$\|\phi\|_{W^{1, \infty}(\Omega)} \le 1$, because otherwise,
\[
E_\infty(u_\infty) \le E_\infty(u_\infty + \phi) + M \|\phi\|_{W^{1, \infty}(\Omega)}^2
\]
for the number $M = E_\infty(u_\infty)$.

We first note that $f_\infty \in W_\loc^{1, q}(\Omega)$ for some $q > 1$
by Lemma \ref{lem:regularity}. Hence the equation $S_{u_\infty}^* f_\infty = 0$
can be written in the form
\[
\div\bigl(A Df_\infty + f_\infty \div A - f_\infty b_z(x, u_\infty, Du_\infty)\bigr) + f_\infty b_y(x, u_\infty, Du_\infty) = 0.
\]
With standard regularity theory for elliptic
equations, we then obtain higher regularity, in particular
$f_\infty \in W_\loc^{2, p}(\Omega)$ for every $p < \infty$,
and the results of
Hardt and Simon \cite{Hardt-Simon:89} on the structure of the nodal set
apply. It follows that $f_\infty \not= 0$ almost everywhere.

If $e_\infty = 0$, then \eqref{eq:Euler-Lagrange1} implies that $E_\infty(u_\infty) = 0$, and $u_\infty$ is in fact a global minimiser.
Thus it suffices to consider $e_\infty > 0$.

Fix $\phi \in \W_0^{2, \infty}(\Omega)$ with $\|\phi\|_{W^{1, \infty}(\Omega)} \le 1$. For $t \in \R$, note that
\[
\frac{\partial}{\partial t} S(u_\infty + t\phi) = S_{u_\infty + t\phi}' \phi
\]
and
\[
\begin{split}
\frac{\partial^2}{\partial t^2} S(u_\infty + t\phi) & = \phi^2 b_{yy}(x, u_\infty + t\phi, Du_\infty + tD\phi) \\
& \quad + 2 \phi D\phi \cdot b_{yz}(x, u_\infty + t\phi, Du_\infty + tD\phi) \\
& \quad + (D\phi \otimes D\phi) : b_{zz}(x, u_\infty + t\phi, Du_\infty + tD\phi).
\end{split}
\]
Thus Taylor's theorem, applied to the function $t \mapsto S(u_\infty + t\phi)(x)$ for each $x \in \Omega$,
implies there exists a function $\tau \colon \Omega \to [0, 1]$ such that
\begin{equation} \label{eqn:Taylor}
S(u_\infty + \phi) = S(u_\infty) + S_{u_\infty}' \phi + B_\tau
\end{equation}
almost everywhere in $\Omega$, where
\begin{multline*}
B_\tau = \frac{1}{2} \phi^2 b_{yy}(x, u_\infty + \tau \phi, Du_\infty + \tau D\phi) + \phi D\phi \cdot b_{yz}(x, u_\infty + \tau \phi, Du_\infty + \tau D\phi) \\
+ \frac{1}{2} (D\phi \otimes D\phi) : b_{zz}(x, u_\infty + \tau \phi, Du_\infty + \tau D\phi).
\end{multline*}
Therefore,
\begin{equation} \label{eq:S-squared}
(S(u_\infty + \phi))^2 = (S(u_\infty))^2 + 2 \bigl(S(u_\infty) + B_\tau\bigr)S_{u_\infty}' \phi +  (S_{u_\infty}' \phi)^2 + 2S(u_\infty) B_\tau
+ B_\tau^2.
\end{equation}

Formula \eqref{eqn:Taylor} implies that $B_\tau$ is measurable. Since $\|\phi\|_{W^{1, \infty}(\Omega)} \le 1$,
we have the estimate
\[
\|u_\infty + \tau \phi\|_{L^\infty(\Omega)} + \|Du_\infty + \tau D\phi\|_{L^\infty(\Omega)} \le C_1
\]
for a constant $C_1$ that is independent of $\phi$. Hence there exists a constant $C_2$, also independent of $\phi$, such that
\begin{equation} \label{eq:B-estimate}
\|B_\tau\|_{L^\infty(\Omega)} \le C_2 \|\phi\|_{W^{1, \infty}(\Omega)}^2.
\end{equation}

We now claim that
\begin{equation} \label{eq:almost-mimimiser}
E_\infty(u_\infty + \phi) \ge E_\infty(u_\infty) - 2C_2 \|\phi\|_{W^{1, \infty}(\Omega)}^2.
\end{equation}
Once this inequality is established, the proof is complete.

If $2C_2 \|\phi\|_{W^{1, \infty}(\Omega)}^2 > E_\infty(u_\infty)$, then \eqref{eq:almost-mimimiser} is obvious.
Thus we assume that $2C_2 \|\phi\|_{W^{1, \infty}(\Omega)}^2 \le E_\infty(u_\infty)$.

If $S_{u_\infty}' \phi = 0$ almost everywhere, then \eqref{eq:S-squared}
and \eqref{eq:B-estimate} imply that
\[
(S(u_\infty + \phi))^2 \ge (S(u_\infty))^2+ 2 S(u_\infty) B_\tau \ge (S(u_\infty))^2 - 2C_2 E_\infty(u_\infty) \|\phi\|_{W^{1, \infty}(\Omega)}^2
\]
almost everywhere. Therefore,
\begin{equation} \label{eq:pseudominimiser}
(E_\infty(u_\infty + t\phi))^2 \ge (E_\infty(u_\infty))^2 - 2C_2 E_\infty(u_\infty) \|\phi\|_{W^{1, \infty}(\Omega)}^2.
\end{equation}

If $S_{u_\infty}' \phi \neq 0$ in a set of positive measure, then we test equation \eqref{eq:Euler-Lagrange2} with $\phi$.
(This is possible in view of Lemma \ref{lem:distributional-solutions}.) We obtain
\[
\int_\Omega f_\infty \, S_{u_\infty}' \phi \, dx = 0.
\]
Recall that $f_\infty \neq 0$ almost everywhere.
Therefore, there exists a set $\Omega_+ \subseteq \Omega$ of positive measure such
that $f_\infty \, S_{u_\infty}' \phi > 0$ in $\Omega_+$. As $f_\infty$ has the same sign as $S(u_\infty)$
almost everywhere by \eqref{eq:Euler-Lagrange1},
this means that $S(u_\infty) \, S_{u_\infty}' \phi > 0$ almost everywhere in $\Omega_+$.
Equation \eqref{eq:Euler-Lagrange1} also implies that $|S(u_\infty)| = e_\infty$ almost everywhere. As
$\|\phi\|_{W^{1, \infty}(\Omega)} \le \sqrt{e_\infty/(2C_2)}$ by the above assumption, 
inequality \eqref{eq:B-estimate} implies that $S(u_\infty) + B_\tau$ has the same sign as $S(u_\infty)$ almost everywhere. Hence
\[
\bigl(S(u_\infty) + B_\tau\bigr) S_{u_\infty}' \phi > 0
\]
almost everywhere in $\Omega_+$. With the help of \eqref{eq:S-squared} and
\eqref{eq:B-estimate}, we conclude that
\[
(S(u_\infty + \phi))^2 \ge e_\infty^2 - 2C_2 e_\infty \|\phi\|_{W^{1, \infty}(\Omega)}^2
\]
in $\Omega_+$. Hence we obtain \eqref{eq:pseudominimiser} in this case as well.

Finally, from \eqref{eq:pseudominimiser} we now obtain the estimate
\[
\begin{split}
E_\infty(u_\infty + \phi) & \ge E_\infty(u_\infty)\sqrt{1 - \frac{2C_2 \|\phi\|_{W^{1, \infty}(\Omega)}^2}{E_\infty(u_\infty)}} \\
& \ge E_\infty(u_\infty)\left(1 - \frac{2C_2 \|\phi\|_{W^{1, \infty}(\Omega)}^2}{E_\infty(u_\infty)}\right) \\
& = E_\infty(u_\infty) - 2C_2 \|\phi\|_{W^{1, \infty}(\Omega)}^2.
\end{split}
\]
This proves \eqref{eq:almost-mimimiser} and completes the proof.
\end{proof}

\section{Necessity of the equations}

In this section, we prove Theorem \ref{thm:necessity}. For this purpose, we
require the following lemma, which is an extension of a result proved by the
authors in a previous paper \cite[Lemma 8]{Katzourakis-Moser:19}. This is
where the extra regularity assumptions in Theorem \ref{thm:necessity} are
used. For $r > 0$, we use the notation
$\Omega_r = \set{x \in \Omega}{\dist(x, \partial \Omega) < r}$ here.

\begin{lemma} \label{lem:boundary_behaviour}
Suppose that $\partial \Omega$ is of class $C^3$ and $u_0 \in C^2(\overline{\Omega})$.
Let $g \in C^0(\overline{\Omega})$ and $u \in C^1(\overline{\Omega})$.
Given $\epsilon > 0$, there exist $r > 0$ and $v \in C^2(\overline{\Omega})$,
with $v = u_0$ and $Dv = Du_0$ on $\partial \Omega$, such that
$\|S_u' v - g\|_{L^\infty(\Omega_r)} \le \epsilon$.
\end{lemma}

\begin{proof}
Let $\delta > 0$.
Choose $\tilde{u}_0 \in C^4(\overline{\Omega})$ such that
\[
\|u_0 - \tilde{u}_0\|_{C^2(\overline{\Omega})} \le \delta
\]
and choose $\tilde{g} \in C^2(\overline{\Omega})$ such that
\[
\bigl\|g - \tilde{g} - u_0 b_y(x, u, Du) - Du_0 \cdot b_z(x, u, Du)\bigr\|_{C^0(\overline{\Omega})} \le \delta.
\]

Consider a number $r_0 > 0$ such that the function $x \mapsto \dist(x, \partial \Omega)$ is of class $C^3$
in $\overline{\Omega}_{2r_0}$. We can construct a function $\rho \in C^3(\overline{\Omega})$ such that
$\rho(x) = \dist(x, \partial \Omega)$ for $x \in \Omega_{r_0}$.
Note that there exists $c > 0$ such that $A : D\rho \otimes D\rho \ge c$ in $\Omega_{r_0}$ by the fact that
$A$ is uniformly positive definite and $|D\rho| = 1$. Define
$\lambda \in C^2(\overline{\Omega})$ such that
$\lambda = 2A : D\rho \otimes D\rho$ in $\Omega_{r_0}$ and $\lambda > 0$ everywhere. Now define
\[
h = \tilde{g} - A : D^2 \tilde{u}_0
\]
and
\[
v = u_0 + \frac{\rho^2 h}{\lambda}.
\]
Then
\[
Dv = Du_0 + \frac{2\rho h}{\lambda} D\rho + \rho^2 D\left(\frac{h}{\lambda}\right)
\]
and
\begin{multline*}
D^2 v = D^2 u_0 + \frac{2h}{\lambda} D\rho \otimes D\rho \\
+ \frac{2\rho h}{\lambda} D^2 \rho + 2\rho D\rho \otimes D\left(\frac{h}{\lambda}\right) + 2\rho D\left(\frac{h}{\lambda}\right) \otimes D\rho + \rho^2 D^2 \left(\frac{h}{\lambda}\right).
\end{multline*}
Thus there exist $\Phi \in C^0(\overline{\Omega}; \R^n)$ and $\Psi \in C^0(\overline{\Omega}; \R^{n \times n})$ such that
\[
Dv = Du_0 + \rho \Phi
\]
and
\[
D^2 v = D^2 u_0 + \frac{2h}{\lambda} D\rho \otimes D\rho + \rho \Psi.
\]
In $\Omega_{r_0}$, it follows that
\[
\begin{split}
S_u' v & = A : D^2u_0 + \frac{2h}{\lambda} A : D\rho \otimes D\rho + \rho A : \Psi \\
& \quad + \left(u_0 + \frac{\rho^2h}{\lambda}\right) b_y(x, u, Du) + \left(Du_0 + \rho \Phi\right) \cdot b_z(x, u, Du) \\
& = A : D^2 u_0 + \tilde{g} - A : D^2 \tilde{u}_0  + u_0 b_y(x, u, Du) + Du_0 \cdot b_z(x, u, Du) + \rho X,
\end{split}
\]
where
\[
X = A : \Psi + \frac{\rho h}{\lambda} b_y(x, u, Du) + \Phi \cdot b_z(x, u, Du).
\]
By the choice of $\tilde{u}_0$ and $\tilde{g}$, we conclude that
\[
\|S_u' v - g\|_{L^\infty(\Omega_r)} \le (\|A\|_{L^\infty(\Omega)} + 1) \delta + r\|X\|_{L^\infty(\Omega)}.
\]
Choosing $\delta$ and $r$ sufficiently small, we obtain the desired inequality.

The boundary conditions are readily checked as well.
\end{proof}

\begin{proof}[Proof of Theorem \ref{thm:necessity}]
We assume that $\partial \Omega$ is of class $C^3$ and $u_0 \in C^2(\overline{\Omega})$. We further assume that $S$ is admissible.
Suppose that $u_\infty \in u_0 + \W_0^{2, \infty}(\Omega)$ is an almost-minimiser
of $E_\infty$. We wish to show that there exist $e_\infty \ge 0$ and
$f_\infty \in L^1(\Omega) \setminus \{0\}$ such that \eqref{eq:Euler-Lagrange1}
is satisfied almost everywhere and \eqref{eq:Euler-Lagrange2} weakly.

If $E_\infty(u_\infty) = 0$, then we choose
$e_\infty = 0$. By a consequence of the Fredholm alternative \cite[Theorem 6.2.4]{Evans:98}, we can either find a nontrivial
solution of $S_{u_\infty}^* f_\infty = 0$ in $W_0^{1, 2}(\Omega)$, or we can solve the boundary value problem
\begin{alignat*}{2}
S_{u_\infty}^* f_\infty & = 0 & \quad & \text{in $\Omega$}, \\
f_\infty & = 1 && \text{on $\partial \Omega$},
\end{alignat*}
in $W^{1, 2}(\Omega)$.
In either case, we conclude that $f_\infty \in L^1(\Omega)$ and does not vanish identically. Hence the required conditions are satisfied.

We now assume that
$E_\infty(u_\infty) > 0$. Because $u_\infty$ is an almost-minimiser,
there exists $M \in \R$ such that
\[
E_\infty(u_\infty) \le E_\infty(v) + M \|u_\infty - v\|_{W^{1, \infty}(\Omega)}^2
\]
for all $v \in u_0 + \W_0^{2, \infty}(\Omega)$. Choose $p_0 > n$ such
that the statement from Definition \ref{def:admissible} applies.
Note that by \cite[Lemma 9.17]{Gilbarg-Trudinger:83}, there exists
a constant $C$ such that for all $\phi \in W_0^{2, p_0}(\Omega)$,
the inequality
\[
\|\phi\|_{W^{2, p_0}(\Omega)} \le C\|A : D^2\phi\|_{L^{p_0}(\Omega)}
\]
holds true. In conjunction with the Sobolev embedding theorem, this implies
that there exists $\mu > 0$ such that
\[
E_\infty(u_\infty) \le E_\infty(v) + \mu \|A : D^2(u_\infty - v)\|_{L^{p_0}(\Omega)}^2
\]
under the above assumptions.

For $p < \infty$, we consider the functionals
\[
E_p(u) = \left(\fint_\Omega |S(u)|^p \, dx\right)^{\frac{1}{p}}.
\]
Furthermore, we fix $\sigma > 0$ and define
\[
E_p^\sigma(u) = E_p(u) + \sigma \|A : D^2(u_\infty - u)\|_{L^{p_0}(\Omega)}^2.
\]
For every $p \ge p_0$, choose a minimiser $u_p \in u_0 + W_0^{2, p}(\Omega)$ of $E_p^\sigma$.
(This can be found with the direct method. Coercivity of the functional is
a consequence of the assumption that $S$ is admissible, similarly to the proof of Proposition \ref{prp:minimisers}.
Lower semicontinuity of $E_p$ with respect to weak convergence in $W^{2, p}(\Omega)$ is also proved
analogously to Proposition \ref{prp:minimisers}, and the lower semicontinuity of the additional term
follows from its convexity.)
For $p_0 \le p \le q$, the minimality of $E_p(u_p)$ and H\"older's inequality
imply that
\begin{equation} \label{eqn:monotonicity}
E_p^\sigma(u_p) \le E_p^\sigma(u_q) \le E_q^\sigma(u_q) \le E_q^\sigma(u_\infty) = E_q(u_\infty) \le E_\infty(u_\infty).
\end{equation}
Because $S$ is admissible, we infer that
\begin{equation} \label{eqn:u_p-estimate}
\limsup_{p \to \infty} \|u_p\|_{W^{2, q}(\Omega)} < \infty
\end{equation}
for any $q < \infty$. We may therefore choose a sequence $p_k \to \infty$
and find $w_\infty \in u_0 + \bigcap_{q < \infty} W_0^{2, q}(\Omega)$ such
that $u_{p_k} \rightharpoonup w_\infty$
weakly in $W^{2, q}(\Omega)$ for every $q < \infty$. Then we also have the
strong convergence $u_{p_k} \to w_\infty$ in $W^{1, \infty}(\Omega)$.
It further follows that $S(u_{p_k}) \rightharpoonup S(w_\infty)$
weakly in $L^q(\Omega)$ for every $q < \infty$.

By the lower semicontinuity of the $L^q$-norm with respect to weak convergence and
by \eqref{eqn:monotonicity}, we have the inequalities
\begin{equation} \label{eq:lower-semicontinuity}
\begin{split}
E_\infty^\sigma(w_\infty) & = \lim_{q \to \infty} E_q^\sigma(w_\infty) \\
& \le \limsup_{q \to \infty} \liminf_{k \to \infty} E_q^\sigma(u_{p_k}) \\
& \le \liminf_{k \to \infty} E_{p_k}^\sigma(u_{p_k}) \\
& \le E_\infty(u_\infty).
\end{split}
\end{equation}
Hence $w_0 \in u_0 + \W_0^{2, \infty}(\Omega)$. As $u_\infty$ is an almost-minimiser of $E_\infty$, we also know that
\[
E_\infty(u_\infty) \le E_\infty^\mu(w_\infty).
\]
Hence
\[
(\sigma - \mu) \|A : D^2(w_\infty - u_\infty)\|_{L^{p_0}(\Omega)} \le 0,
\]
which implies that $w_\infty = u_\infty$ if $\sigma > \mu$. As this fixes the limit, we can in fact conclude that
$u_p \rightharpoonup u_\infty$, as $p \to \infty$, weakly in $W^{2, q}(\Omega)$ for every $q < \infty$.
It also follows from \eqref{eq:lower-semicontinuity} and \eqref{eqn:monotonicity} that $E_p^\sigma(u_p) \to E_\infty(u_\infty)$ as $p \to \infty$.

We further observe that for any sequence $p_k \to \infty$,
\begin{multline*}
E_\infty(u_\infty) + \sigma \liminf_{k \to \infty} \|A : D^2(u_{p_k} - u_\infty)\|_{L^{p_0}(\Omega)}^2 \\
\begin{aligned}
& = \lim_{q \to \infty} E_q(u_\infty) + \sigma \liminf_{k \to \infty} \|A : D^2(u_{p_k} - u_\infty)\|_{L^{p_0}(\Omega)}^2 \\
& \le \limsup_{q \to \infty} \liminf_{k \to \infty} E_q(u_{p_k}) + \sigma \liminf_{k \to \infty} \|A : D^2(u_{p_k} - u_\infty)\|_{L^{p_0}(\Omega)}^2 \\
& \le \limsup_{q \to \infty} \liminf_{k \to \infty} E_q^\sigma(u_{p_k}).
\end{aligned}
\end{multline*}
For every fixed $q < \infty$, H\"older's inequality gives
$E_q^\sigma(u_{p_k}) \le E_{p_k}^\sigma(u_{p_k})$ whenever $k$ is sufficiently large. Hence
\[
\liminf_{k \to \infty} E_q^\sigma(u_{p_k}) \le \liminf_{k \to \infty} E_{p_k}^\sigma(u_{p_k}),
\]
and we conclude that
\[
E_\infty(u_\infty) + \sigma \liminf_{k \to \infty} \|A : D^2(u_{p_k} - u_\infty)\|_{L^{p_0}(\Omega)}^2 \le \liminf_{k \to \infty} E_{p_k}^\sigma(u_{p_k}) = E_\infty(u_\infty).
\]
It follows that
\[
\lim_{p \to \infty} \|A : D^2(u_p - u_\infty)\|_{L^{p_0}(\Omega)} = 0.
\]
Set $e_\infty = E_\infty(u_\infty)$ and $e_p = E_p(u_p)$. Then it follows that
\[
e_\infty = \lim_{p \to \infty} e_p.
\]
Set furthermore
\[
a_p = \|A : D^2(u_p - u_\infty)\|_{L^{p_0}(\Omega)}.
\]
Then the Euler-Lagrange equation for $u_p$ is
\begin{multline*}
e_p^{1 - p} S_{u_p}^* \bigl(|S(u_p)|^{p - 2} S(u_p)\bigr) \\
+ 2\sigma |\Omega| a_p^{2 - p_0} \div \div\bigl(|A : D^2(u_p - u_\infty)|^{p_0 - 2} (A : D^2(u_p - u_\infty)) A\bigr) = 0.
\end{multline*}
Set
\[
f_p = e_p^{1 - p} |S(u_p)|^{p - 2} S(u_p)
\]
and
\[
\phi_p = a_p^{2 - p_0} |A : D^2(u_p - u_\infty)|^{p_0 - 2} A : D^2(u_p - u_\infty).
\]
Then we have the system
\begin{align}
|f_p|^{\frac{p - 2}{p - 1}} S(u_p)  & = e_p f_p, \label{eqn:Euler-Lagrange_p1} \\
S_{u_p}^* f_p + 2\sigma |\Omega| \div \div (\phi_p A) & = 0. \label{eqn:Euler-Lagrange_p2}
\end{align}

We compute
\begin{equation} \label{eqn:L^{p'}-estimate}
\fint_\Omega |f_p|^{p/(p - 1)} \, dx = e_p^{-p} \fint_\Omega |S(u_p)|^p \, dx = 1.
\end{equation}
Hence we can find a sequence $p_k \to \infty$ such that $f_{p_k}$ converges, in the weak* sense in $(C^0(\overline{\Omega}))^*$, to a Radon measure $F_\infty$ on $\overline{\Omega}$. Testing equation \eqref{eqn:Euler-Lagrange_p2} with $\eta \in C_0^\infty(\Omega)$, we see that
\[
\begin{split}
\int_\Omega S_{u_\infty}' \eta \, dF_\infty & = \lim_{k \to \infty} \int_\Omega f_{p_k} S_{u_\infty}' \eta \, dx \\
& = \lim_{k \to \infty} \int_\Omega f_{p_k} S_{u_{p_k}}' \eta \, dx \\
& = - 2\sigma |\Omega| \lim_{k \to \infty} \int_\Omega \phi_{p_k} A : D^2 \eta \, dx.
\end{split}
\]
(In the second step, we have used the fact that $b_y(x, u_p, Du_p) \to b_y(x, u_\infty, Du_\infty)$ and $b_z(x, u_p, Du_p) \to b_z(x, u_\infty, Du_\infty)$
uniformly as $p \to \infty$.)
If $p_0'$ is the exponent conjugate to $p_0$, then
\begin{equation} \label{eq:convergence-phi_p}
\|\phi_p\|_{L^{p_0'}(\Omega)} = a_p^{2 - p_0} \left(\int_\Omega |A : D^2(u_p - u_\infty)|^{p_0} \, dx\right)^{\frac{p_0 - 1}{p_0}} = a_p \to 0
\end{equation}
as $p \to \infty$. Hence $F_\infty$ is a distributional solution of
$S_{u_\infty}^* F_\infty = 0$. According to Lemma \ref{lem:regularity},
its restriction to $\Omega$ is absolutely continuous with respect to the Lebesgue measure, and the Radon-Nikodym derivative
$f_\infty \in L^1(\Omega)$ belongs to $W_\loc^{1, q}(\Omega)$ for all $q \in (1, \frac{n}{n - 1})$ (but $F_\infty$ may have a part supported on $\partial \Omega$ as well). Obviously, we now have equation \eqref{eq:Euler-Lagrange2}
in the weak sense.

Set $h_p = f_p + 2\sigma |\Omega| \phi_p$. Then \eqref{eqn:Euler-Lagrange_p2}
can be written in the form
\[
S_{u_p}^* h_p = 2\sigma |\Omega| \left(\phi_p b_y(x, u_p, Du_p) - \div \bigl(\phi_p b_z(x, u_p, Du_p)\bigr)\right).
\]
We already know that
the functions $b_z(x, u_p, Du_p)$ and $b_y(x, u_p, Du_p)$ are uniformly
bounded in $L^\infty(\Omega)$. Because of \eqref{eq:convergence-phi_p},
we have uniform bounds for $\phi_p b_y(x, u_p, Du_p)$ and for $\phi_p b_z(x, u_p, Du_p)$ in $L^{p_0'}(\Omega)$. The coefficients of $S_{u_p}^*$ are also bounded in
$L^\infty(\Omega)$.
According to Lemma \ref{lem:regularity}, this means that we have a uniform bound
for $\|h_p\|_{W^{1, p_0'}(\Omega')}$ for any precompact open set
$\Omega' \Subset \Omega$. In particular, we can choose the above
sequence $p_k \to \infty$ such that $(h_{p_k})_{k \in \N}$
converges in $L^{p_0'}(\Omega')$. Since we have the convergence $\phi_p \to 0$
in the same space by
\eqref{eq:convergence-phi_p}, we must have $L^{p_0'}$-convergence
for $(f_{p_k})_{k \in \N}$. The limit is of course $f_\infty$.

We may assume that $f_{p_k} \to f_\infty$ almost everywhere in $\Omega'$ and,
at the same time, that $\|f_{p_k} - f_\infty\|_{L^{p_0'}(\Omega')} \le 2^{-k}$
for every $k \in \N$ (otherwise we choose a further subsequence with this
property). If we set
\[
\theta = |f_\infty| + \sum_{k = 1}^\infty |f_{p_k} - f_\infty|,
\]
then this guarantees that $\theta \in L^{p_0'}(\Omega')$. Moreover, we see
that $|f_{p_k}| \le \theta$ for every $k \in \N$. We therefore obtain the pointwise inequality
\[
\left| |f_{p_k}|^{\frac{p_k - 2}{p_k - 1}} - |f_\infty| \right|^{p_0'} \le \left(1 + 2\theta\right)^{p_0'}
\]
(provided that $p_k > 2$), and the dominated convergence theorem implies that
\[
|f_{p_k}|^{\frac{p_k - 2}{p_k - 1}} \to |f_\infty|
\]
in $L^{p_0'}(\Omega')$.
Since we have the weak convergence of $S(u_p)$ to $S(u_\infty)$ in
$L^q(\Omega)$ for any $q < \infty$, we conclude that
\[
|f_{p_k}|^{\frac{p_k - 2}{p_k - 1}} S(u_{p_k}) \rightharpoonup |f_\infty| S(u_\infty)
\]
weakly in $L_\loc^1(\Omega)$. Recall that $e_p \to e_\infty$ as $p \to \infty$. Thus passing to the
limit in \eqref{eqn:Euler-Lagrange_p1}, we obtain \eqref{eq:Euler-Lagrange1}.

It remains to show that $f_\infty \not\equiv 0$.

For any $p \ge p_0$, we compute
\[
|\Omega| e_p = e_p^{1 - p} \int_\Omega |S(u_p)|^p \, dx = \int_\Omega f_p S(u_p) \, dx = \int_\Omega f_p\left(S_{u_p}' u_p - g_p\right) \, dx,
\]
where
\[
g_p = - b(x, u_p, Du_p) + u_p b_y(x, u_p, Du_p) + Du_p \cdot b_z(x, u_p, Du_p).
\]
We define $g_\infty$ by the analogous formula as well.
Given $\epsilon > 0$, we choose $r > 0$ and $v \in u_0 + \W_0^{2, \infty}(\Omega)$ with
\[
\|S_{u_\infty}'v - g_\infty\|_{L^\infty(\Omega_r)} \le \epsilon
\]
with the help of Lemma \ref{lem:boundary_behaviour}. Then,
by \eqref{eqn:Euler-Lagrange_p2},
\[
\begin{split}
|\Omega| e_p & = \int_\Omega f_p\left(S_{u_p}' u_p - g_p\right) \, dx \\
& = \int_\Omega f_p\left(S_{u_p}' v - g_p\right) \, dx - 2\sigma |\Omega| \int_\Omega \phi_p A : D^2(u_p - v) \, dx \\
& = \int_{\Omega_r} f_p\left(S_{u_\infty}' v - g_\infty\right) \, dx + \int_{\Omega_r} f_p\left(S_{u_p}' v - S_{u_\infty}' v - g_p + g_\infty\right) \, dx \\
& \quad + \int_{\Omega \setminus \Omega_r} f_p\left(S_{u_p}' v - g_p\right) \, dx - 2\sigma |\Omega| \int_\Omega \phi_p A : D^2(u_p - v) \, dx \\
& \le \epsilon |\Omega| + \int_{\Omega_r} f_p\left(S_{u_p}' v - S_{u_\infty}' v - g_p + g_\infty\right) \, dx \\
& \quad + \int_{\Omega \setminus \Omega_r} f_p\left(S_{u_p}' v - g_p\right) \, dx + 2 \sigma |\Omega| a_p \, \|A : D^2(u_p - v)\|_{L^{p_0}(\Omega)}.
\end{split}
\]
Letting $p \to \infty$, we note that
$S_{u_p}' v \to S_{u_\infty}' v$ and $g_p \to g_\infty$ uniformly in $\Omega$,
while
\[
\limsup_{p \to \infty} \|f_p\|_{L^1(\Omega)} < \infty
\]
by \eqref{eqn:L^{p'}-estimate}. Furthermore, we know that $a_p \to 0$, while $A : D^2 u_p$ is uniformly bounded in $L^{p_0}(\Omega)$. We further know that
$f_{p_k} \to f_\infty$ in $L^1(\Omega \setminus \Omega_r)$. Hence
\[
|\Omega| (e_\infty - \epsilon) \le \int_{\Omega \setminus \Omega_r} f_\infty (S_{u_\infty}' v - g_\infty) \, dx.
\]
Choosing $\epsilon < e_\infty$, we conclude that the integral on the right-hand side does not vanish. Hence $f_\infty \not\equiv 0$,
and this concludes the proof.
\end{proof}

\section{Examples}

The condition for admissible operators in Definition \ref{def:admissible} often follows from standard estimates
for linear operators. For semilinear ones, some additional arguments are sometimes required.
In this section, we consider two examples. We restrict our attention to operators
of the form
\[
S(u) = \Delta u + g(u)
\]
for some function $g \in C^2(\R)$ here. We show that $S$ is admissible if
either $g$ has the correct sign or satisfies a suitable growth condition.

\begin{proposition}
If $y g(y) \le 0$ for all $y \in \R$, then $S$ is admissible.
\end{proposition}

\begin{proof}
Choose any $p_0 > n/2$. Given $u_0 \in \W^{2, \infty}(\Omega)$, let
\[
\alpha = \sup_{x \in \partial \Omega} |u_0(x)|.
\]
Suppose that $u \in u_0 + \W_0^{2, \infty}(\Omega)$ satisfies
$\|S(u)\|_{L^p(\Omega)} \le \Lambda$ for some $p \ge p_0$. Then
$\|S(u)\|_{L^{p_0}(\Omega)} \le |\Omega|^{1/p_0 -1/p} \Lambda$ by
H\"older's inequality.
We now solve the boundary value problem
\begin{alignat*}{2}
\Delta w & = -|S(u)| & \quad & \text{in $\Omega$}, \\
w & = \alpha && \text{on $\partial \Omega$}.
\end{alignat*}
Then $w \ge \alpha$ by the maximum principle. Furthermore, using $L^p$-estimates
for the Laplacian and the Sobolev embedding theorem, we see that $w$ is bounded by
a constant that depends only on $n$, $\Omega$, $p_0$, $u_0$, and $\Lambda$.

In the set
$\set{x \in \Omega}{u(x) > 0}$, we have the inequality
\[
\Delta u = S(u) - g(u) \ge S(u) \ge \Delta w.
\] 
Hence
the comparison principle implies that $u \le w$ in $\Omega$.

Similarly, we show that $u$ is bounded from below by a constant depending only on $n$, $\Omega$, $p_0$, $u_0$, and $\Lambda$.
The condition from Definition \ref{def:admissible} then follows with standard elliptic estimates.
\end{proof}

For our second example, we assume that
\begin{equation} \label{eq:growth-implicit}
\lim_{y \to \pm \infty} \frac{yg(y)}{\int_0^y g(t) \, dt} = \alpha.
\end{equation}
That is, the function $g$ has asymptotic growth at $\pm \infty$ like $y \mapsto cy^{\alpha - 1}$ for some $c \in \R$.

\begin{proposition}
Let $n \ge 3$. Suppose that $g$ satisfies \eqref{eq:growth-implicit}
for some $\alpha \in [2, \frac{2n}{n - 2})$. Then $S$ is admissible.
\end{proposition}

\begin{proof}
Define
\[
G(y) = \int_0^y g(t) \, dt.
\]
Given $\beta > \alpha$, inequality \eqref{eq:growth-implicit} implies that
$y G'(y) < \beta G(y)$ when $|y|$ is sufficiently large. The Gr\"onwall
inequality implies that there exists
$C_1 > 0$ with
\[
G(y) \le C_1 (|y|^\beta + 1)
\]
for all $y \in \R$. Using \eqref{eq:growth-implicit} again, we find another
constant $C_2$ such that
\begin{equation} \label{eq:growth-explicit}
|g(y)| \le C_2(|y|^{\beta - 1} + 1)
\end{equation}
for all $y \in \R$.

Now suppose that $u_0 \in \W^{2, \infty}(\Omega)$ and consider
$u \in u_0 + \W_0^{2, \infty}(\Omega)$. We write $\nu$ for the outer normal
vector on $\partial \Omega$ and $\sigma$ for the surface measure on
$\partial \Omega$. Then an integration by parts yields the
identity
\begin{equation} \label{eq:energy1}
\int_\Omega (|Du|^2 - ug(u)) \, dx = \int_{\partial \Omega} u_0 \nu \cdot Du_0 \, d\sigma - \int_\Omega u S(u) \, dx.
\end{equation}
We furthermore compute
\[
\div \left((x \cdot Du) Du - \left(\frac{1}{2} |Du|^2 - G(u)\right) x\right) = (x \cdot Du) S(u) - \frac{n - 2}{2} |Du|^2 + n G(u).
\]
Hence
\begin{multline} \label{eq:energy2}
\int_\Omega \left(\frac{n - 2}{2} |Du|^2 - n G(u)\right) \, dx = \int_\Omega (x \cdot Du) S(u) \, dx \\
- \int_{\partial \Omega} \left((x \cdot Du_0) (\nu \cdot Du_0) - \left(\frac{1}{2} |Du_0|^2 - G(u_0)\right) x \cdot \nu \right) \, d\sigma.
\end{multline}
Fix $\beta \in (\alpha, \frac{2n}{n - 2})$. Then the combination of
\eqref{eq:energy1} and \eqref{eq:energy2} implies that
\begin{multline} \label{eq:example2}
\left(\frac{1}{\beta} - \frac{n - 2}{2n}\right) \int_\Omega |Du|^2 \, dx = \int_\Omega \left(\frac{u g(u)}{\beta} - G(u)\right) \, dx \\
- \int_\Omega \left(\frac{1}{\beta} uS(u) + \frac{1}{n} (x \cdot Du) S(u)\right) \, dx + \frac{1}{\beta} \int_{\partial \Omega} u_0 \nu \cdot Du_0 \, d\sigma \\
+ \frac{1}{n} \int_{\partial \Omega} \left((x \cdot Du_0) (\nu \cdot Du_0) - \left(\frac{1}{2} |Du_0|^2 - G(u_0)\right) x \cdot \nu \right) \, d\sigma.
\end{multline}

Under the assumptions of the proposition, we know that $yg(y) \le \beta G(y)$ whenever $|y|$ is sufficiently large.
Hence there exists some constant $C_3$, depending only on $g$, such that
\[
\int_\Omega \left(\frac{u g(u)}{\beta} - G(u)\right) \, dx \le C_3 |\Omega|.
\]
The boundary integrals in \eqref{eq:example2} depend only on $n$, $\Omega$,
$u_0$, $g$, and $\beta$. Hence there exists $C_4 = C_4(n, \Omega, u_0, g, \alpha)$
such that
\[
\int_\Omega |Du|^2 \, dx \le C_4 \left(\|S(u)\|_{L^2(\Omega)} \|u\|_{W^{1, 2}(\Omega)} + 1\right).
\]
From this and the Poincar\'e inequality, we derive the estimate
\[
\|u\|_{W^{1, 2}(\Omega)} \le C_5\left(\|S(u)\|_{L^2(\Omega)} + 1\right)
\]
for a constant $C_5$ with the same dependence.

Fix $p_0 > n$ and consider $p \ge p_0$. If we assume that
\[
\|S(u)\|_{L^p(\Omega)} \le \Lambda,
\]
then we can now use standard bootstrapping arguments to derive higher estimates.
Since the growth of $g$ described in \eqref{eq:growth-explicit} is
subcritical for the purpose of such estimates, we will eventually obtain
a bound for $\|u\|_{W^{2, p}(\Omega)}$ that depends only on $n$, $\Omega$,
$u_0$, $g$, $p$, and $\Lambda$. Hence $S$ is admissible.
\end{proof}

\def\cprime{$'$}
\providecommand{\bysame}{\leavevmode\hbox to3em{\hrulefill}\thinspace}
\providecommand{\MR}{\relax\ifhmode\unskip\space\fi MR }
\providecommand{\MRhref}[2]{%
  \href{http://www.ams.org/mathscinet-getitem?mr=#1}{#2}
}
\providecommand{\href}[2]{#2}

\end{document}